\documentclass[12pt,reqno]{amsart}

\usepackage{amssymb}

\usepackage[all]{xy}

\numberwithin{equation}{section}

\newtheorem{theorem}{Theorem}[section]
\newtheorem{proposition}[theorem]{Proposition}
\newtheorem{lemma}[theorem]{Lemma}

\theoremstyle{definition}

\begin{document}

\title[Quot schemes and Fourier-Mukai transformation]{Quot schemes
and Fourier-Mukai transformation}

\author[I. Biswas]{Indranil Biswas}

\address{Mathematics Department, Shiv Nadar University, NH91, Tehsil
Dadri, Greater Noida, Uttar Pradesh 201314, India}

\email{indranil.biswas@snu.edu.in, indranil29@gmail.com}

\author[U. Dubey]{Umesh V Dubey}

\address{Harish-Chandra Research Institute, Chhatnag Road, Jhunsi,
Prayagraj 211019, India}

\email{umeshdubey@hri.res.in}

\author[M. Kumar]{Manish Kumar}

\address{Statistics and Mathematics Unit, Indian Statistical Institute,
Bangalore 560059, India}

\email{manish@isibang.ac.in}

\author[A.J. Parameswaran]{A. J. Parameswaran}

\address{School of Mathematics, Tata Institute of Fundamental
Research, Homi Bhabha Road, Bombay 400005, India}

\email{param@math.tifr.res.in}



\begin{abstract}
We consider several related examples of Fourier-Mukai transformations involving the quot 
scheme. A method of showing conservativity of these Fourier-Mukai transformations is
described.
\end{abstract}

\maketitle

\section{Introduction}

Fourier-Mukai transformations arise in numerous contexts in algebraic geometry \cite{Hu}, \cite{BBH}.
Over time it has emerged to be an immensely useful concept. Here we investigate
Fourier-Mukai transformations in a particular context, namely in the set-up of quot schemes.
We show that the conservativity of the Fourier-Mukai transformation holds in the following cases:
\begin{enumerate}
\item Let $M$ be an irreducible smooth projective variety over an algebraically closed field $k$
such that $\dim M \, \geq\, 2$. Denote by $\text{Hilb}^d(M)$ the Hilbert scheme parametrizing the
zero-dimensional subschemes of $M$ of length $d$. Let $P_M$ (respectively, $P_H$) be the
projection to $M$ (respectively, $\text{Hilb}^d(M)$) of the tautological subscheme
${\mathcal S}\, \subset\, M\times \text{Hilb}^d(M)$. 
Let $E$ and $F$ be two vector bundles on $M$ such that the vector bundles $P_{H*}P^*_M E$ and
$P_{H*}P^*_M F$ on ${\rm Hilb}^d(M)$ are isomorphic. Then we show that $E$ and $F$ are isomorphic
(see Proposition \ref{prop1}).

\item Let $Q_M$ (respectively, $Q_S$) be the projection to $M$ (respectively, $\text{Sym}^d(M)$) 
of the tautological subscheme ${\mathbb S}\, \subset\, M\times\text{Sym}^d(M)$. 
Let $E$ and $F$ be two vector bundles on $M$ such that the vector bundles $Q_{S*}Q^*_M E$ and
$Q_{S*}Q^*_M F$ on ${\rm Sym}^d(M)$ are isomorphic. Then we show that $E$ and $F$ are isomorphic
(see Lemma \ref{lem1}).

\item Let $C$ be an irreducible smooth projective curve defined over $k$.
Fix a vector bundle $E$ over $C$ of rank at least two.
Let ${\mathcal Q}^d(E)$ denote the quot scheme parametrizing the torsion quotients of $E$
of degree $d$. There is a tautological quotient
$\Phi^*_C E \, \longrightarrow\, {\mathbf Q}$ over $C\times{\mathcal Q}^d(E)$, where
$\Phi_C\, :\, C\times {\mathcal Q}^d(E) \, \longrightarrow\, C$ is the natural projection.
Let $V$ and $W$ be vector bundles on $C$ such that the vector bundles
$\Phi_{Q*}((\Phi^*_C V)\otimes {\mathbf Q})$ and $\Phi_{Q*}((\Phi^*_C W)\otimes {\mathbf Q})$
on ${\mathcal Q}^d(E)$ are isomorphic, where $\Phi_Q\, :\, C\times
{\mathcal Q}^d(E)\, \longrightarrow\, {\mathcal Q}^d(E)$ is the natural projection.
Then we show that $E$ and $F$ are isomorphic (see Proposition \ref{prop2}).
\end{enumerate}

We also prove a similar result in the context of vector bundles on curves equipped with a
group action (see Section \ref{se3}).

A key method in our proofs is the Atiyah's Krull--Schmidt theorem for vector bundles.

\section{A Fourier-Mukai transformation}

\subsection{Vector bundles on Hilbert schemes}

Let $k$ be an algebraically closed field. Let $M$ be an irreducible smooth projective variety over $k$
such that $\dim M \, \geq\, 2$.
For any integer $d\, \geq\, 2$, let $\text{Hilb}^d(M)$ denote the Hilbert scheme parametrizing the
zero-dimensional subschemes of $M$ of length $d$. We have the natural projections
$$
M \, \xleftarrow{\,\,\,p_M\,\,}\, M\times \text{Hilb}^d(M) \, \xrightarrow{\,\,\,p_H\,\, }\, \text{Hilb}^d(M).
$$
There is a tautological subscheme
$$
{\mathcal S}\, \subset\, M\times \text{Hilb}^d(M)
$$
such that for any $z\, \in\, \text{Hilb}^d(M)$ the preimage $p^{-1}_H(z)$ is the subscheme $z\, \subset\, M$. The restriction
of $p_M$ (respectively, $p_H$) to $\mathcal S$ will be denoted by $P_M$ (respectively, $P_H$).

For any vector bundle $E$ on $M$ we have the direct image $P_{H*}P^*_M E$ on $\mathcal S$. We note that
$P_{H*}P^*_M E$ is locally free because $P_H$ is a finite morphism and $P^*_M E$ is
locally free. Let
\begin{equation}\label{f1}
{\widetilde E}\, :=\, P_{H*}P^*_M E
\end{equation}
be this vector bundle; its rank is $d\cdot\text{rank}(E)$.
It is known that two vector bundles $E$ and $F$ on $M$ are isomorphic if $\widetilde E$ and
$\widetilde F$ are isomorphic \cite{KR}, \cite{BN}. We will give a very simple proof of it.

\begin{proposition}\label{prop1}
Let $E$ and $F$ be two vector bundles on $M$ such that the corresponding vector bundles $\widetilde E$ and
$\widetilde F$ on ${\rm Hilb}^d(M)$ are isomorphic (see \eqref{f1}). Then $E$ and $F$ are isomorphic.
\end{proposition}

\begin{proof}
Since $\text{rank}(\widetilde{E})$ and $\text{rank}(\widetilde{F})$ are $d\cdot\text{rank}(E)$
and $d\cdot\text{rank}(F)$ respectively, it follows that $\text{rank}(E)\,=\, \text{rank}(F)$. Let
$\text{rank}(E)\,=\, r\,=\, \text{rank}(F)$.

Fix a zero dimensional subscheme $Z^0\, \subset\, M$ of length $d-1$. Let
$Z^0_{\rm red} \, =\, \{x_1,\, \cdots,\, x_b\}\, \subset\, M$ be the reduced subscheme for $Z^0$. The complement
$M \setminus Z^0_{\rm red} \,=\, M\setminus \{x_1,\, \cdots,\, x_b\}$ will be denoted by $M^0$. Let
\begin{equation}\label{e2}
\iota \ :\, M^0\, \longrightarrow\, M
\end{equation}
be the inclusion map. We have a morphism
$$
\varphi\, :\, M^0\, \longrightarrow\, \text{Hilb}^d(M)
$$
that sends any $x\, \in\, M^0$ to $Z_0 \cup \{x\}$. The pullback $\varphi^*{\widetilde E}$
(respectively, $\varphi^*{\widetilde F}$) is isomorphic to $\iota^*E\oplus V_0$ (respectively,
$\iota^*F\oplus V_0$), where $V_0$ is a trivial vector bundle on
$M^0$ of rank $(d-1)r$ and $\iota$ is the map in \eqref{e2}. The vector bundles
$\varphi^*{\widetilde E}$ and $\varphi^*{\widetilde F}$ are isomorphic because
${\widetilde E}$ and ${\widetilde F}$ are isomorphic. So 
\begin{equation}\label{e3}
\iota^*E\oplus V_0\, =\, \iota^*F\oplus V_0 .
\end{equation}

There are no nonconstant functions on $M^0$ (recall
that $\dim M\, \geq \, 2$). Hence using \cite[p.~315, Theorem 2(i)]{At}, from \eqref{e3} it follows that
$\iota^*E\,=\, \iota^*F$ (see \cite{Su} for vast generalizations of \cite{At}). Hence we have
$$
\iota_*\iota^*E\,=\, \iota_*\iota^*F\, .
$$
But $\iota_*\iota^*E$ (respectively, $\iota_*\iota^*F$) is $E$ (respectively, $F$).
This completes the proof.
\end{proof}

The line of arguments in Proposition \ref{prop1} works in some other contexts. We will describe two such
instances.

\subsection{Vector bundles on symmetric product}\label{se2.2}

As before, $M$ is an irreducible smooth projective variety of dimension at least two.
For any integer $d\, \geq\, 2$, let $\text{Sym}^d(M)$ denote the quotient of $M^d$ under the action of
the symmetric group $S_d$ that permutes the factors of the Cartesian product. We recall that $\text{Sym}^d(M)$ is
a normal projective variety. There is a tautological subscheme
\begin{equation}\label{f2}
{\mathbb S}\, \subset\, M\times\text{Sym}^d(M)
\end{equation}
parametrizing all $(z,\, y)\, \in\,M\times\text{Sym}^d(M)$ such that $z\, \in\, y$. Let
$$
Q_M\, :\, {\mathbb S}\, \longrightarrow\, M \ \ \text{ and }\ \ Q_S\, :\, {\mathbb S}\, \longrightarrow\, \text{Sym}^d(M)
$$
be the natural projections. For any vector bundle $E$ on $M$ the direct image
$$
{\widehat E}\, :=\, Q_{S*}Q^*_M E
$$
on $\text{Sym}^d(M)$ is locally free because $Q_S$ is a finite morphism and $Q^*_M E$
is locally free.

\begin{lemma}\label{lem1}
Let $E$ and $F$ be two vector bundles on $M$ such that the corresponding vector bundles $\widehat E$ and
$\widehat F$ on ${\rm Sym}^d(M)$ are isomorphic. Then $E$ and $F$ are isomorphic.
\end{lemma}

\begin{proof}
Fix any $z_0 \, =\, \{x_1,\, \cdots,\, x_{d-1}\} \, \in \, \text{Sym}^{d-1}(M)$ (repetitions are allowed).
Let
$$
\iota\, :\, M^0\,:=\, M\setminus z_0\, \hookrightarrow\, M
$$
be the inclusion map. We have a morphism
$$
\phi\, :\, M\setminus z_0\, \longrightarrow\, \text{Sym}^d(M),\ \ \ x\, \longmapsto\, \{x,\, z\}.
$$
First note that $\phi^*{\widehat E}\,=\, \phi^*{\widehat F}$ because ${\widehat E}\,=\,
{\widehat F}$. Evidently, we have
$\phi^*{\widehat E}\,=\, (\iota^* E)\oplus {\mathcal O}^{\oplus (d-1)\cdot {\rm rank}(E)}_{M^0}$ and
$\phi^*{\widehat F}\,=\, (\iota^* F)\oplus {\mathcal O}^{\oplus (d-1)\cdot {\rm rank}(F)}_{M^0}$.
Now the argument in the proof of Proposition \ref{prop1} goes through without any changes.
\end{proof}

\subsection{Vector bundles on quot scheme}\label{sec2.3}

Let $C$ be an irreducible smooth projective curve defined over $k$. Fix a vector bundle $E$ over $C$ of rank at least two.
Fix an integer $d\, \geq\, 1$. Let ${\mathcal Q}^d(E)$ denote the quot scheme parametrizing the torsion quotients of $E$
of degree $d$. Let
\begin{equation}\label{j1}
\Phi_C\, :\, C\times {\mathcal Q}^d(E) \, \longrightarrow\, C \ \ \text{ and }\ \ \Phi_Q\, :\, C\times
{\mathcal Q}^d(E)\, \longrightarrow\, {\mathcal Q}^d(E)
\end{equation}
be the natural projections. There is a tautological quotient
\begin{equation}\label{y1}
\Phi^*_C E \, \longrightarrow\, {\mathbf Q}
\end{equation}
over $C\times{\mathcal Q}^d(E)$ whose restriction to any $C\times \{Q\}$, where $Q\, \in\, {\mathcal Q}^d(E)$, is the
quotient of $E$ represented by $Q$.

Given a vector bundle $V$ on $C$, we have the direct image
\begin{equation}\label{e4}
F(V)\, :=\, \Phi_{Q*}((\Phi^*_C V)\otimes {\mathbf Q})\, \longrightarrow\, {\mathcal Q}^d(E),
\end{equation}
where $\Phi_Q$ and $\Phi_C$ are the projections in \eqref{j1}, and $\mathbf Q$ is the quotient in \eqref{y1}; this $F(V)$
is a vector bundle because the support of $\mathbf Q$ is finite over ${\mathcal Q}^d(E)$.

\begin{proposition}\label{prop2}
Let $V$ and $W$ be vector bundles on $C$ such that the corresponding vector bundles $F(V)$ and $F(W)$ are
isomorphic (see \eqref{e4}). Then $V$ and $W$ are isomorphic.
\end{proposition}

\begin{proof}
Since $\text{rank}(F(V))$ and $\text{rank}(F(W))$ are $d\cdot\text{rank}(V)$ and $d\cdot\text{rank}(W)$ respectively,
from the given condition that $F(V)$ and $F(W)$ are isomorphic we conclude that $\text{rank}(V)\,=\, \text{rank}(W)$.
Let $r$ denote $\text{rank}(V)\,=\, \text{rank}(W)$.

Let
\begin{equation}\label{e5}
\beta\,:\, {\mathbb P}(E)\, \longrightarrow\, X
\end{equation}
be the projective bundle parametrizing the hyperplanes in the fibers of $E$. So ${\mathbb P}(E)\,=\, {\mathcal Q}^1(E)$. For any
$z\, \in\, \beta^{-1}(x)\, \subset\, {\mathbb P}(E)$, if $H(z)\, \subset\, E_x$ is the corresponding hyperplane, then the element
of ${\mathcal Q}^1(E)$ for $z$ represents the quotient sheaf $E \, \longrightarrow\, E_x/H(z)$ of $E$. For any $z\, \in\,
{\mathbb P}(E)$, the quotient sheaf map from $E$ to the torsion quotient $E_x/H(z)$ of $E$ of degree $1$ corresponding to $z$ will be denoted by $\textbf{z}$.

Fix $d-1$ distinct points $x_1,\, \cdots ,\, x_{d-1}$ of $X$. Fix points $y_i\, \in\, \beta^{-1}(x_i)$, $1\, \leq\, i\, \leq\, d-1$,
where $\beta$ is the projection in \eqref{e5}. The complement ${\mathbb P}(E)\setminus \{y_1,\, \cdots,\, y_{d-1}\}$ will be denoted
by ${\mathcal P}$. Let
\begin{equation}\label{e6}
\iota\, :\, {\mathcal P} \, \hookrightarrow\, {\mathbb P}(E)
\end{equation}
be the inclusion map.

Note that the subset $\{y_1,\, \cdots,\, y_{d-1}\}$ defines a point of ${\mathcal Q}^{d-1}(E)$
representing the quotient $\bigoplus_{j=1}^{d-1} \textbf{y}_j$ of $E$; this point of
${\mathcal Q}^{d-1}(E)$ will be denoted by $\textbf{y}$. We have a morphism
$$
\Psi\, :\, {\mathcal P}\, \longrightarrow\, {\mathcal Q}^d(E),\, \ \ z\, \longmapsto\, \textbf{z}\oplus \textbf{y};
$$
recall that ${\mathbb P}(E)\,=\, {\mathcal Q}^1(E)$ and both $\textbf{y}$ and $\textbf{z}$ are quotients of $E$.

Now the vector bundle $\Psi^*F(V)$ (respectively, $\Psi^*F(W)$) is isomorphic to 
$(\iota^*((\beta^*V)\otimes {\mathcal O}_{{\mathbb P}(E)}(1)))\oplus A$ (respectively,
$(\iota^*((\beta^*W)\otimes {\mathcal O}_{{\mathbb P}(E)}(1)))\oplus A$), where $\iota$ and 
$\beta$ are the maps in \eqref{e6} and \eqref{e5} respectively, and $A$ is a trivial vector 
bundle on $\mathcal P$ of rank $r(d-1)$; the tautological line bundle on ${\mathbb P}(E)$
is denoted by ${\mathcal O}_{{\mathbb P}(E)}(1)$.

Since $V$ and $W$ are isomorphic we conclude that
$(\iota^*((\beta^*V)\otimes {\mathcal O}_{{\mathbb P}(E)}(1)))\oplus A$ and
$(\iota^*((\beta^*V)\otimes {\mathcal O}_{{\mathbb P}(E)}(1)))\oplus A$ are isomorphic.
As there are no nonconstant functions on $\mathcal P$ it follows that 
$(\iota^*\beta^*V)\otimes {\mathcal O}_{{\mathbb P}(E)}(1)$ and
$(\iota^*\beta^*W)\otimes {\mathcal O}_{{\mathbb P}(E)}(1)$ are isomorphic.
This implies that $\iota^*\beta^*V$ and $\iota^*\beta^*W$ are isomorphic.

The direct image $\iota_*\iota^*\beta^*V$ (respectively, $\iota_*\iota^*\beta^*W$) is
$\beta^*V$ (respectively, $\beta^*W$). Hence we conclude that $\beta^*V$ and $\beta^*W$ are isomorphic.
So $\beta_*\beta^*V\,=\, V$ is isomorphic to $\beta_*\beta^*W\,=\, W$.
\end{proof}

\section{Action of group on a curve}\label{se3}

Let $C$ be an irreducible smooth projective curve, and let $\Gamma$ be a finite group
acting faithfully on $C$. Consider the quotient curve
\begin{equation}\label{e7}
f\, :\, C\, \longrightarrow\, Y \,:=\, C/\Gamma .
\end{equation}
For any vector bundle $V$ on $Y$, the pullback $f^*V$ is a $\Gamma$-equivariant vector bundle on $C$.

The order of the group $\Gamma$ is denoted by $d$. We have a morphism
\begin{equation}\label{e8}
\rho\, :\, Y\, \longrightarrow\, \text{Sym}^d(C)
\end{equation}
that sends any $y\, \in\, Y$ to the element of $\text{Sym}^d(C)$ given by the scheme-theoretic inverse
image $f^{-1}(y)$, where $f$ is the map in \eqref{e7}. To describe $\rho$ explicitly, let
$\{z_1,\, z_2,\, \cdots,\, z_n\}$ be the reduced inverse image $f^{-1}(y)_{\rm red}$. Then
$$
\rho(y)\, =\, \sum_{i=1}^n b_iz_i,
$$
where $b_i$ is the order of the isotropy subgroup $\Gamma_{z_i}\, \subset\, \Gamma$
of $z_i$ for the action of $\Gamma$ on $C$. Note that $\rho$ is an embedding.

The action of $\Gamma$ on $C$ produces an action of $\Gamma$ on $\text{Sym}^d(C)$. The action of any $\gamma
\, \in\, \Gamma$ sends any $(x_1,\, \cdots, \, x_d)\, \in\, \text{Sym}^d(C)$ to $(\gamma(x_1),\, \cdots, \, \gamma(x_d))$.
We have
\begin{equation}\label{f3}
\rho(Y)\, \subset\, \text{Sym}^d(C)^\Gamma .
\end{equation}

We note that $\text{Sym}^d(C)$ is an irreducible smooth projective variety of dimension $d$. As in \eqref{f2},
\begin{equation}\label{f4}
{\mathbb S}\, \subset\, C\times\text{Sym}^d(C)
\end{equation}
is the tautological subscheme parametrizing all $(c,\, x) \, \in\, C\times\text{Sym}^d(C)$ such that $c\,\in\, x$.
Let 
\begin{equation}\label{f5}
Q_C\, :\, {\mathbb S}\, \longrightarrow\, C \ \ \text{ and }\ \ Q_S\, :\, {\mathbb S}\, \longrightarrow\, \text{Sym}^d(C)
\end{equation}
be the natural projections. For any vector bundle $E$ on $C$ of rank $r$, the direct image
\begin{equation}\label{e9}
{\widehat E}\, :=\, Q_{S*}Q^*_C E
\end{equation}
is a vector bundle on $\text{Sym}^d(C)$ of rank $dr$.

We will describe an alternative construction of the vector bundle ${\widehat E}$ in \eqref{e9}.
For $1\, \leq\, i\, \leq\, d$, let
$$
p_i\, :\, C^d\, \longrightarrow\, C
$$
be the projection to the $i$-th factor. Let
\begin{equation}\label{e10}
P\, :\, C^d\, \longrightarrow\, \text{Sym}^d(C)
\end{equation}
be the quotient map for the action of the symmetric group $S_d$ that permutes the factors of $C^d$. The action of $S_d$ on $C^d$
lifts to the vector bundle
$$
E^{[d]}\, :=\, \bigoplus_{i=1}^d p^*_iE \, \longrightarrow\, C^d.
$$
The action of $S_d$ on $E^{[d]}$ produces an action of $S_d$ on $P_*E^{[d]}$, where $P$ is the projection in \eqref{e10}.
The vector bundle ${\widehat E}$ in \eqref{e9} coincides with the $S_d$--invariant part
$$
(P_*E^{[d]})^{S_d}\, \subset\, P_*E^{[d]}.
$$

The actions of $\Gamma$ on $C$ and $\text{Sym}^d(C)$ (see \eqref{f3}) together produce a diagonal action of $\Gamma$ on
$C\times \text{Sym}^d(C)$. This action of $\Gamma$ on $C\times\text{Sym}^d(C)$ preserves the subscheme $\mathbb S$
in \eqref{f4}. For this action of $\Gamma$ on $\mathbb S$, the projections $Q_C$ and $Q_S$ in \eqref{f5}
are evidently $\Gamma$--equivariant.

Now let $E$ be a $\Gamma$--equivariant vector bundle on $C$. Since the projections $Q_C$ and $Q_S$ in \eqref{f5}
are $\Gamma$--equivariant, the vector bundle ${\widehat E}$ in \eqref{e9} is also $\Gamma$--equivariant. From
\eqref{f3} it now follows that the vector bundle
\begin{equation}\label{e11}
\rho^*{\widehat E} \, \longrightarrow\, Y
\end{equation}
is equipped with an action of $\Gamma$ over the trivial action of $\Gamma$ on $Y$.

\begin{proposition}\label{prop3}
Let $E$ and $F$ be vector bundles on $Y$ such that the corresponding $\Gamma$--equivariant vector bundles
$\rho^*\widehat{f^*E}$ and $\rho^*\widehat{f^*F}$ on $Y$ are isomorphic. Then $E$ and $F$ are isomorphic.
\end{proposition}

\begin{proof}
The vector bundle $f^*E$ has a natural action of $\Gamma$ because it is pulled back from
$C/\Gamma$. The action of $\Gamma$ on $f^*E$ produces an action of $\Gamma$ on $f_*f^*E$ over the
trivial action of $\Gamma$ on $Y$. Similarly, $\Gamma$ acts on $f_*f^*F$.

Consider $Q^{-1}_S(\rho(Y))\, \subset\, \mathbb{S}$, where $Q_S$ and $\rho$ are the maps in \eqref{f5} and
\eqref{e8} respectively. Let
$$
Q'_C\, :=\, Q_C\big\vert_{Q^{-1}_S(\rho(Y))}\, :\, Q^{-1}_S(\rho(Y))\, \longrightarrow\, C
$$
be the restriction of the map $Q_C$ in \eqref{f5}. It is straightforward to check that this map $Q'_C$ is an isomorphism.
So we have the commutative diagram
\begin{equation}\label{e12}
\begin{matrix}
C & \xleftarrow{\,\,\,Q'_C\,\,} & Q^{-1}_S(\rho(Y))\\
f\Big\downarrow\,\, && Q_S\Big\downarrow\,\,\,\,\\
Y & \xrightarrow{\,\,\,\, \rho\,\,} & \rho(Y)
\end{matrix}
\end{equation}
where the horizontal maps are isomorphisms. Moreover, all the maps in \eqref{e12} are $\Gamma$--equivariant
with $\Gamma$ acting trivially on $Y$ and $\rho(Y)$. Therefore, from \eqref{e12} we conclude that there are isomorphisms
\begin{equation}\label{e13}
f_*f^*E \, \stackrel{\sim}{\longrightarrow}\, \rho^*\widehat{f^*E} \ \ \text{ and }\ \
f_*f^*F \, \stackrel{\sim}{\longrightarrow}\, \rho^*\widehat{f^*F}
\end{equation}
as $\Gamma$--equivariant vector bundles.

Since the $\Gamma$--equivariant vector bundles
$\rho^*\widehat{f^*E}$ and $\rho^*\widehat{f^*F}$ are isomorphic, from \eqref{e13} it follows that
\begin{equation}\label{j11}
f_*f^*E \, \, \stackrel{\sim}{\longrightarrow}\,\, f_*f^*F
\end{equation}
as $\Gamma$--equivariant vector bundles

Next we will show that
\begin{equation}\label{j12}
(f_*f^*E)^\Gamma \,=\, E \ \ \text{ and }\ \ (f_*f^*F)^\Gamma \,=\, F.
\end{equation}

To prove \eqref{j12}, first note that the action of $\Gamma$ on $C$ produces an
action of $\Gamma$ on $f_*{\mathcal O}_C$. The projection formula gives that
$$
f_*f^*E\,\, \stackrel{\sim}{\longrightarrow}\,\, E\otimes (f_*{\mathcal O}_C).
$$
The action of $\Gamma$ on $f_*{\mathcal O}_C$ and the trivial action of $\Gamma$ on $E$
together produce an action of $\Gamma$ on $E\otimes (f_*{\mathcal O}_C)$. The above isomorphism
between $f_*f^*E$ and $E\otimes (f_*{\mathcal O}_C)$ is evidently $\Gamma$--equivariant. Since
$(f_*{\mathcal O}_C)^\Gamma\,=\, {\mathcal O}_Y$, we conclude that \eqref{j12} holds.

Finally, the proposition follows from \eqref{j11} and \eqref{j12}.
\end{proof}

\section{Alternative constructions}

Let $C$ be a smooth projective curve over $k$ and $E$ a vector bundle on $C$. Unlike in Section \ref{sec2.3},
$E$ can be
a line bundle; we no longer assume $\text{rank}(E)$ to be at least two. As before, ${\mathcal Q}^d(E)$ denotes the quot scheme
that parametrizes the torsion quotients of $E$ of degree $d$. Let
\begin{equation}\label{ga}
\gamma\, :\, {\mathcal Q}^d(E)\, \longrightarrow\, \text{Sym}^d(C)
\end{equation}
be the natural Chow morphism.

For any vector bundle $V$ on $C$, consider the vector bundle $F(V)$ on ${\mathcal Q}^d(E)$ constructed in \eqref{e4}. We will
describe its direct image $\gamma_*F(V)$ on $\text{Sym}^d(C)$, where $\gamma$ is the map in \eqref{ga}.

For every $1\, \leq\, j\, \leq\, d$, let $\varphi_j\, :\, C^d\, \longrightarrow\, C$ be the projection to the
$j$-th factor. Take a vector bundle $V$ on $C$. We have the vector bundle
\begin{equation}\label{V}
{\mathcal V}\, :=\, \bigoplus_{j=1}^d \varphi^*_j (V\otimes E)\, \longrightarrow\, C^d\, .
\end{equation}
The symmetric group $S_d$ acts on $C^d$ by permuting the factors of the tensor product (see Section \ref{se2.2}). The corresponding
quotient is $\text{Sym}^d(C)$. As in \eqref{e10}, let 
\begin{equation}\label{P}
P\, :\, C^d\, \longrightarrow\, C^d/S_d\,=\, \text{Sym}^d(C)
\end{equation}
be the quotient map. The action of $S_d$ on $C^d$ has a natural lift to an action of $S_d$ on the vector bundle $\mathcal V$ in \eqref{V}.
This action of $S_d$ on $\mathcal V$ produces an action of $S_d$ on the direct image $P_*{\mathcal V}$, where $P$ is the projection
in \eqref{P}.

\begin{lemma}\label{lem2}
The direct image $\gamma_*F(V)$ on $\text{Sym}^d(C)$, where $F(V)$ and $\gamma$ are as in \eqref{e4} and \eqref{ga} respectively,
is naturally identified with the $S_d$--invariant part
$$
(P_*{\mathcal V})^{S_d}\, \subset\,P_*{\mathcal V}
$$
for the above action of $S_d$ on $P_*{\mathcal V}$.
\end{lemma}

\begin{proof}
There is a natural homomorphism
$$
\varpi\, :\, F(V)\, \longrightarrow\, P_*{\mathcal V}\, .
$$
It is straightforward to check that $\varpi(F(V))\, \subset\,
(P_*{\mathcal V})^{S_d}\, \subset\,P_*{\mathcal V}$ and that the resulting
homomorphism $F(V)\, \longrightarrow\,(P_*{\mathcal V})^{S_d}$ is an isomorphism.
\end{proof}

Let
\begin{equation}\label{h1}
\Psi_C\, :\, C\times \text{Sym}^d(C) \, \longrightarrow\, C \ \ \text{ and }\ \ \Psi_S\, :\, C\times
\text{Sym}^d(C) \, \longrightarrow\, \text{Sym}^d(C)
\end{equation}
be the natural projections. 

Consider $(\text{Id}_C \times \gamma)_*{\mathbf Q}$ on $C\times\text{Sym}^d(C)$, where $\mathbf Q$ is the sheaf in
\eqref{y1} and $\gamma$ is the map in \eqref{ga}. Given a vector bundle $V$ on $C$, we have the direct image
\[
G(V) \,:= \, \Psi_{S*}(\Psi_C^* V \otimes (\text{Id}_C \times \gamma)_*{\mathbf Q}) \,\longrightarrow\, \text{Sym}^d(C),
\]
where $\Psi_C$ and $\Psi_S$ are projections in \eqref{h1}.
 
\begin{proposition}
For any vector bundle $V$ on $C$ there is a natural isomorphism
\[
\gamma_*F(V) \,\simeq\, G(V),
\]
where $F(V)$ is constructed in \eqref{e4}.
\end{proposition}

\begin{proof}
Consider the following commutative diagram
\[
\xymatrix{
&& C \times {\mathcal Q}^d(E) \ar[dll]_-{\Phi_C} \ar[r]^{\Phi_Q} \ar[d]_{(\text{Id}_C \times \gamma)}
& {\mathcal Q}^d(E) \ar[d]_\gamma\\
 C && \ar[ll]^-{\Psi_C} C \times \text{Sym}^d(C) \ar[r]_{\Psi_S} & \text{Sym}^d(C).\\
}
\]
Now using above commutative diagram we can get the required isomorphism as follows
\begin{eqnarray*}
\gamma_* F(V) &\,=\, & \gamma_* \Phi_{Q*}((\Phi^*_C V)\otimes {\mathbf Q})\\
& \simeq\, & \Psi_{S*} (\text{Id} \times \gamma)_* ((\Phi^*_C V)\otimes {\mathbf Q})\\
&=\, & \Psi_{S*} (\text{Id} \times \gamma)_* ((\text{Id} \times \gamma)^*(\Psi^*_C V)\otimes {\mathbf Q}) \\
& \simeq\, & \Psi_{S*} ((\Psi^*_C V)\otimes (\text{Id} \times \gamma)_* {\mathbf Q}) \,~~ \text{(projection formula)}\\
&= \,& G(V). 
\end{eqnarray*}
\end{proof}

\section*{Mandatory declarations}

There is no conflict of interests regarding this manuscript. No funding was received for this manuscript.

\end{document}